\documentclass[a4paper]{amsart}

\ifx\MyBeamer\undefined
\usepackage[breaklinks,colorlinks]{hyperref}
\else
\if\MyBeamer t
\usepackage[breaklinks,colorlinks]{hyperref}
\else
\def\emph{\alert}
\fi
\fi

\usepackage{amsmath}
\usepackage{amssymb}
\usepackage{mathrsfs}
\usepackage{amsthm}
\usepackage{aliascnt}
\usepackage[utf8]{inputenc}
\usepackage[T1]{fontenc}
\usepackage[nofancy]{svninfo}
\ifx\FrenchText\undefined
\usepackage[british]{babel}
\else
\usepackage[british,french]{babel}
\AtBeginDocument{%
\catcode`\:=12%
\catcode`\!=12%
}
\fi

\makeatletter

\newcounter{quotethmcnt}

\ifx\MyBeamer\undefined

\def\equationautorefname~#1\null{(#1)}
\def\itemautorefname~#1\null{#1}
\fi

\ifx\MyBeamer\undefined

\ifx\thmnum\undefined
\newtheorem{thmnum}{}[section]
\fi

\newcommand{\mynewthm}[3][]{%
  \newaliascnt{#2}{thmnum}%
  \newtheorem{#2}[#2]{#3}%
  \aliascntresetthe{#2}%
  \newtheorem*{#2*}{#3}%
  \expandafter\newcommand\csname #2autorefname\endcsname{#3}%
  \expandafter\renewcommand\csname the#2\endcsname{\thethmnum}%
}

\newtheorem*{clm}{Claim}
\newenvironment{clmprf}{%
  \begin{proof}[Proof of claim]%
  }{\end{proof}}

\else

\let\xxx=\frametitle
\def\frametitle#1{%
  \xxx{%
    \setbeamercolor*{math text}{use={titlelike,my math text},fg=titlelike.fg!80!my math text.fg}%
    #1}%
  \setbeamercolor{math text}{use=my math text,fg=my math text.fg}%
}

\newcommand{\beamerenv}[3]{%
\newenvironment<>{#1}%
{%
  \setbeamercolor{temp}{fg=structure.fg}%
  \setbeamercolor{structure}{fg=#2}%
  \setbeamercolor{block body}{use=structure,bg=structure.fg!5!white}%
  \begin{#3}%
}%
{\end{#3}\setbeamercolor{structure}{fg=temp.fg}}}

\newcommand{\mynewthm}[3][green!50!black]{%
  \newtheorem*{#2x}{#3}%
  \beamerenv{#2}{#1}{#2x}%
}

\beamerenv{other}{violet}{block}

\fi

\ifx\FrenchText\undefined
\newcommand{\myiffrench}[2]{#2}
\else
\newcommand{\myiffrench}[2]{\iflanguage{french}{#1}{#2}}
\fi

\theoremstyle{plain}
\mynewthm[red]{thm}{\myiffrench{Théorème}{Theorem}}
\mynewthm{prp}{Proposition}
\mynewthm[purple]{lem}{\myiffrench{Lemme}{Lemma}}
\mynewthm[orange]{fct}{\myiffrench{Fait}{Fact}}
\mynewthm[purple!50!white]{cor}{\myiffrench{Corollaire}{Corollary}}

\theoremstyle{definition}
\mynewthm[green!80!black]{dfn}{\myiffrench{Définition}{Definition}}
\mynewthm{hyp}{\myiffrench{Hypothèse}{Hypothesis}}
\mynewthm{conv}{Convention}
\mynewthm{conj}{Conjecture}
\mynewthm{ntn}{Notation}
\mynewthm{cst}{Construction}

\theoremstyle{remark}
\mynewthm[yellow!90!black]{rmk}{\myiffrench{Remarque}{Remark}}
\mynewthm{qst}{Question}
\mynewthm{exc}{\myiffrench{Exercice}{Exercise}}
\mynewthm{exm}{\myiffrench{Exemple}{Example}}

\ifx\MyBeamer\undefined

\newcommand{\myenumlabel}[1]{\textnormal{(\roman{#1})}}

\fi

\newcounter{cycprfcnt}
\newcounter{cycprffirst}
\newcommand{\cycprfpreamble}[1]%
{%
  \setcounter{cycprfcnt}{1}
  \setcounter{cycprffirst}{#1}
  \setlength{\itemindent}{0.5\leftmargin}%
  \setlength{\leftmargin}{0pt}%
  \newcommand{\cpcurr}{\myenumlabel{cycprfcnt}}%
  \newcommand{\cpnext}{\addtocounter{cycprfcnt}{1}\cpcurr}%
  \newcommand{\impnext}{\cpcurr{} $\Longrightarrow$ \cpnext.}%
  \def\makelabel##1{\ifnum\value{cycprffirst}=0\hspace{-0.7\itemindent}\setcounter{cycprffirst}{1}\fi##1}%
}%

{\begin{list}{\impnext}{\cycprfpreamble{#1}}}%
{\qedhere\end{list}}%

\newenvironment{cycprf*}[1][0]%
{\begin{list}{\impnext}{\cycprfpreamble{#1}}}%
{\end{list}}%

\def\indsym#1#2{%
  \setbox0=\hbox{$\m@th#1x$}%
  \kern\wd0%
  \hbox to 0pt{\hss$\m@th#1\mid$\hbox to 0pt{$\m@th#1^{#2}$\hss}\hss}%
  \lower.9\ht0\hbox to 0pt{\hss$\m@th#1\smile$\hss}%
  \kern\wd0}

\def\nindsym#1#2{%
  \setbox0=\hbox{$\m@th#1x$}%
  \kern\wd0%
  \hbox to 0pt{\hss$\m@th#1\not$\kern1.4\wd0\hss}
  \hbox to 0pt{\hss$\m@th#1\mid$\hbox to 0pt{$\m@th#1^{#2}$\hss}\hss}%
  \lower.9\ht0\hbox to 0pt{\hss$\m@th#1\smile$\hss}%
  \kern\wd0}

\def\dotminussym#1#2{%
  \setbox0=\hbox{$\m@th#1-$}%
  \kern.5\wd0%
  \hbox to 0pt{\hss\hbox{$\m@th#1-$}\hss}%
  \raise.6\ht0\hbox to 0pt{\hss$\m@th#1.$\hss}%
  \kern.5\wd0}

\renewcommand{\emptyset}{\varnothing}
\renewcommand{\setminus}{\smallsetminus}

\usepackage{eucal}
\usepackage[osf]{mathpazo}

\DeclareMathOperator{\Iso}{Iso}

\DeclareMathOperator{\Homeo}{Homeo}

\newcommand{\bN}{\mathbf{N}}

\makeatother

\usepackage{xcolor}
\definecolor{dark-red}{rgb}{0.5,0.15,0.15}
\definecolor{dark-blue}{rgb}{0.15,0.15,0.5}
\definecolor{medium-blue}{rgb}{0,0,0.5}
\hypersetup{
    colorlinks, linkcolor={dark-red},
    citecolor={dark-blue}, urlcolor={medium-blue}
}

\usepackage{enumitem}
\setlist[enumerate,1]{label=(\roman*), font=\normalfont}

\title{Metrizable universal minimal flows of Polish groups have a comeagre orbit}

\author[I.~Ben~Yaacov]{Ita{\"\i} \textsc{Ben~Yaacov}}
\address{Ita{\"\i} \textsc{Ben Yaacov} \\
  Université Claude Bernard -- Lyon 1 \\
  Institut Camille Jordan, CNRS UMR 5208 \\
  43 boulevard du 11 novembre 1918 \\
  69622 Villeurbanne Cedex \\
  France}
\urladdr{\url{http://math.univ-lyon1.fr/~begnac/}}

\author[J.~Melleray]{Julien \textsc{Melleray}}
\address{Julien \textsc{Melleray} \\
  Université Claude Bernard -- Lyon 1 \\
  Institut Camille Jordan, CNRS UMR 5208 \\
  43 boulevard du 11 novembre 1918 \\
  69622 Villeurbanne Cedex \\
  France}
\email{melleray@math.univ-lyon1.fr}
\urladdr{\url{http://math.univ-lyon1.fr/~melleray/}}

\author[T.~Tsankov]{Todor \textsc{Tsankov}}
\address{Todor \textsc{Tsankov}\\
Institut de Math\'ematiques de Jussieu--PRG \\
Universit\'e Paris Diderot, Case 7012 \\
75205 Paris \sc{cedex} 13 \\
France}
\email{todor@math.univ-paris-diderot.fr}
\urladdr{\url{http://webusers.imj-prg.fr/~todor.tsankov/}}

\date{January 2016}

\newcommand{\actson}{\curvearrowright}
\DeclareMathOperator{\ucb}{UC_b}
\DeclareMathOperator{\cli}{C_L}
\newcommand{\dR}{d_{\mathrm{R}}}

\newcommand{\cl}[2][]{\overline{#2}^{#1}}
\newcommand{\sub}{\subseteq}

\begin{document}
\begin{abstract}
We prove that, whenever $G$ is a Polish group with metrizable universal minimal flow $M(G)$, there exists a comeagre orbit in $M(G)$. It then follows that there exists an extremely amenable, closed, co-precompact subgroup $G^*$ of $G$ such that $M(G) = \widehat{G/G^*}$. 
\end{abstract}

\maketitle
\section{Introduction}

The \emph{universal minimal flow (UMF)} $M(G)$ of a topological group $G$ is a compact dynamical system of $G$ canonically associated with the group which is of great interest in topological dynamics (see the end of the introduction for the precise definitions). For locally compact, non-compact groups, the UMF is never metrizable and cannot be described in a meaningful way. However, for many important non-locally compact groups, it turns out that this flow is metrizable, can be computed, and carries interesting information. For example, the UMF of the unitary group of a separable, infinite-dimensional Hilbert space is a point (Gromov--Milman~\cite{Gromov1983}) (in that case, we say that the group is \emph{extremely amenable}), the UMF of $\Homeo(S^1)$ is $S^1$ (Pestov~\cite{Pestov1998}), and the UMF of the infinite symmetric group is the space of all linear orderings on a countable set (Glasner--Weiss~\cite{Glasner2002a}).

The systematic study of metrizable UMFs was initiated in the influential paper \cite{Kechris2005} by Kechris, Pestov, and Todorcevic, where they established a connection with Ramsey theory and computed many interesting examples. It turns out that all known concrete computations of UMFs can be carried out using a single technique, first employed by Pestov~\cite{Pestov2002a}: isolating a co-precompact, extremely amenable subgroup $G^*$ of the group of interest $G$, and then observing that any minimal subflow of the completion $\widehat{G/G^*}$ must be isomorphic to $M(G)$. (In practice, with a judicious choice of $G^*$, $\widehat{G/G^*}$ is often already minimal and such a choice can always be made; see Corollary~\ref{c:minimal}.) Thus computing a metrizable UMF reduces to finding an extremely amenable, co-precompact subgroup for which a variety of techniques have been developed (see Pestov~\cite{Pestov2006}). In this paper, we show that this approach always works.
\begin{thm}
  \label{t:main1}
  Let $G$ be a Polish group such that $M(G)$ is metrizable. Then there exists a closed, co-precompact, extremely amenable subgroup $G^* \leq G$ such that $M(G) = \widehat{G/G^*}$.
\end{thm}
Of course, conversely, if $G$ is Polish and $G^* \leq G$ is co-precompact, then $\widehat{G/G^*}$ is metrizable.

Our proof is based on the recent work of Melleray, Nguyen Van Thé, and Tsankov \cite{Melleray15}, where the conclusion of Theorem~\ref{t:main1} is proved under the additional hypothesis that $M(G)$ has a comeagre orbit (then $G^*$ can be taken to be the stabilizer of a point in this orbit). Thus the main theorem proved in this paper is the following.
\begin{thm}
  \label{t:main2}
  Let $G$ be a Polish group whose universal minimal flow $M(G)$ is metrizable. Then $M(G)$ has a comeagre orbit.
\end{thm}
This theorem answers a question of Angel, Kechris, and Lyons \cite{Angel2014} (Question~15.2, ``the generic point problem'').

In a work simultaneous with (and independent from) \cite{Melleray15}, Zucker~\cite{Zucker2014p} proved Theorem~\ref{t:main1} for the important special case of \emph{non-archimedean} Polish groups (i.e., groups that admit a basis at the identity consisting of open subgroups). His approach is based on topological properties of the Stone--\v{C}ech compactification of a discrete set (the space of cosets $G/V$ for open subgroups $V$) and combinatorial techniques. As Polish groups do not usually admit open subgroups, his method does not directly generalize. In the present paper, we abstract some of his ideas and combine them with a new, \emph{topometric} approach to prove our main result. We note that, while our proof is based in part on Zucker's ideas, our argument is new even for the non-archimedean case.

In \cite{Melleray15}, many structural results are proved for metrizable UMFs with a comeagre orbit; with Theorem~\ref{t:main2}, now they automatically hold with the only assumption of metrizability.

We note that considering Polish groups in the place of arbitrary topological groups is not a real restriction in our setting: if $M(G)$ is metrizable, then the closure of the image of $G$ in $\Homeo(M(G))$ is a Polish group whose minimal flows are exactly the same as those of $G$, so as long as we are interested only in minimal flows, we can replace one with the other.

It is an interesting question whether there is a converse to Theorem~\ref{t:main2}.
\begin{qst}
  \label{q:converse}
  Suppose that $G$ is a Polish group such that $M(G)$ has a comeagre orbit. Is it true that $M(G)$ is metrizable?
\end{qst}
For the reader who prefers to only deal with metrizable spaces, the question can be rephrased equivalently as follows: if all metrizable $G$-flows have a comeagre orbit, does this imply that $M(G)$ is metrizable?

We conclude the introduction with the definitions of the notions that we have used above. A \emph{$G$-flow} is a continuous action of a topological group $G$ on a compact, Hausdorff space, and a flow is \emph{minimal} if all of its orbits are dense. A fundamental theorem of topological dynamics, due to Ellis, establishes the existence of a unique \emph{universal minimal flow} for any topological group $G$, that is, a minimal $G$-flow which maps continuously and equivariantly onto every minimal $G$-flow. A subgroup $H \leq G$ is called \emph{co-precompact} if the completion of the uniform space $G/H$ (equipped with the quotient of the right uniform structure of $G$) is compact, equivalently, if for any open $V \ni 1_G$, there exists a finite $F \sub G$  such that $VFH = G$.

The paper is organized as follows. We first describe some general properties of the topometric structure on the Samuel compactification of a metric space, then apply those in the dynamical setting to prove the main result, and finish by noting several corollaries. We have tried to keep the paper self-contained: in particular, everything that we need about topometric spaces is proved in the next section.

\subsection*{Acknowledgements} Research on this paper was partially supported by the ANR projects GruPoLoCo (ANR-11-JS01-008) and GAMME (ANR-14-CE25-0004). We are grateful to the anonymous referee for some helpful remarks and for providing a reference.

%%%%%%%%%%%%%%%%%%%%%%%%%%%%%%%%%%%%%%%%%%%%%%%%%%

\section{A topometric structure on the Samuel compactification of a metric space}
Let $(X,d)$ be a bounded metric space. Recall that the \emph{Samuel compactification} $S(X)$ of $X$ is the Gelfand space of the algebra $\ucb(X)$ of all complex-valued, uniformly continuous, bounded functions on $X$. This compactification is characterized by the following universal property: whenever $K$ is a compact Hausdorff space, and $f \colon X \to K$ is uniformly continuous, $f$ extends to a continuous map from $S(X)$ to $K$. The algebra of all continuous bounded functions on $S(X)$ is naturally identified with $\ucb(X)$.

Our aim is to gain a better understanding of compact metrizable subsets of $S(X)$.
Note that, if $d$ is the discrete $0$--$1$ metric, then the Samuel compactification is the Stone--\v{C}ech compactification $\beta X$, which is extremally disconnected, so that its only compact metrizable subsets are finite. This informs our approach here: we wish to give a precise meaning to the intuition that compact metrizable subsets of $S(X)$ are small, and in the discrete case ``small'' means ``finite''. It is then natural, given the approach taken in \cite{BenYaacov2008c}, \cite{BenYaacove} or \cite{BenYaacov15a}, to try to define a topometric structure on $S(X)$ for which ``small'' means ``metrically compact''. This topometric structure was already considered in \cite{BenYaacove}, where it is called the \emph{topometric Stone--\v{C}ech compactification}; below we recall all relevant definitions.

\begin{dfn}
  A \emph{compact topometric space} is a triple $(Z, \tau, \partial)$, where $Z$ is a set, $\tau$ is a compact Hausdorff topology on $Z$, and $\partial$ is a distance on $Z$ such that the following conditions are satisfied:
  \begin{itemize}
  \item the $\partial$-topology refines $\tau$;
  \item $\partial$ is $\tau$-lower semicontinuous, i.e., the set $\{(a,b) \in Z^2 \colon \partial(a,b) \le r \}$ is $\tau$-closed for every $r \geq 0$.
  \end{itemize}
\end{dfn}

We define a topometric structure on $S(X)$ as follows: we let $\tau$ be the Gelfand topology (i.e., the weak$^*$ topology inherited from the dual of $\ucb(X)$) and define $\partial$ by:
\[
  \partial(a,b)= \sup \{|f(a)-f(b)| : f \in \cli(X) \},
\]
where $\cli(X)$ denotes the set of all bounded $1$-Lipschitz functions on $(X, d)$. To make sense of this, one must recall that functions in $\cli(X)$, being bounded and uniformly continuous, uniquely extend to continuous functions on $S(X)$. To see why the $\partial$-topology refines $\tau$, recall that any function in $\ucb(X)$ is a uniform limit of Lipschitz maps, so $\tau$ has a basis of open sets of the form 
\[
  \{a \in S(X) : f_1(a) \in I_1, \ldots, f_n(a) \in I_n \},
\]
where each $f_j$ belongs to $\cli(X)$ and each $I_j$ is an open interval.

As is usual when working with topometric spaces, we follow the convention that topological terms refer to $\tau$, and metric vocabulary refers to $\partial$; however, we will be careful to specify which one of the two we mean whenever there is danger of confusion.

Note that $\partial$ and $d$ coincide on $X$, and that elements of $\cli(X)$ extend to maps on $S(X)$ which are both $\tau$-continuous and $\partial$-$1$-Lipschitz. Note also that if $d$ is the discrete $0$--$1$ distance on $X$, then $\partial$ is the discrete $0$--$1$ distance on $S(X)$.

We will need the fact that the distance $\partial$ is complete on $S(X)$; in fact, this is true for any compact topometric space, as proved in \cite{BenYaacov2008c}.
\begin{lem}
Let $(Z, \tau, \partial)$ be a compact topometric space. Then $\partial$ is complete.
\end{lem}
\begin{proof}
  Let $(z_n)$ be a Cauchy sequence; for all $n$, define $r_n= \sup \{\partial(z_n,z_m) : m \ge n\}$.
Then $r_n$ converges to $0$; let $F_n$ denote the closed ball of radius $r_n$ centred at $z_n$. Each $F_n$ is $\tau$-closed, hence compact, and since $F_n$ contains $z_m$ for all $m \ge n$, this family has the finite intersection property. By compactness, $\bigcap_{n \in \bN} F_n$ is non-empty; it must be a singleton, which is the $\partial$-limit of the sequence $(z_n)$.
\end{proof}

The following theorem generalizes a well-known fact from the discrete case.
\begin{thm}\label{t:caraccompact}
  Let $(X, d)$ be a metric space and let $(S(X), \tau, \partial)$ be the topometric Stone--\v{C}ech compactification of $X$. Then every $\tau$-convergent sequence in $S(X)$ is $\partial$-convergent.
\end{thm}

In order to prove the theorem, we first establish two lemmas. The first is the topometric analogue of the fact that in the discrete case, $S(X)$ is extremally disconnected. If $A$, $B$ are two subsets of a metric space $(Z, d)$, we denote
\begin{equation*}
  d(A, B) = \inf \{d(a, b) : a \in A, b \in B \}.
\end{equation*}
\begin{lem}\label{l:caraccompact1}
Let $U,V$ be $\tau$-open subsets of $S(X)$. Then we have
\[
  \partial(\cl[\tau]{U}, \cl[\tau]{V})= \partial(U, V)= d(U \cap X, V \cap X).
\]
\end{lem}
\begin{proof}
First note that, since $X$ is dense in $S(X)$, we have $\cl[\tau]{U}= \cl[\tau]{U \cap X}$ and $\cl[\tau]{V}= \cl[\tau]{V \cap X}$. Consider the function $f \in \cli(X)$ defined by $f(x) = d(x, U \cap X)$. It extends to a $\tau$-continuous, $\partial$-$1$-Lipschitz map on $S(X)$, which we still denote by $f$. Since $f=0$ on $U \cap X$, we must also have $f=0$ on $\cl[\tau]{U}$ by continuity; similarly, $f \ge d(V \cap X, U \cap X)$ on $V \cap X$, so $f \ge d(V \cap X,U\cap X)$ on $\cl[\tau]{V}$.
Hence $f$ witnesses the fact that $\partial(\cl[\tau]{U}, \cl[\tau]{V}) \ge d(V \cap X, U \cap X)$; since $d(U \cap X,V \cap X)$ is equal to $\partial(U \cap X, V \cap X)$ by the definition of $\partial$, this inequality must in fact be an equality, and we are done.
\end{proof}

% \textcolor{red}{Note that in the argument above one can show (using compactness) that $f(.)= \partial(., \overline{U})$, so $\partial(., \overline{U})$ is continuous. I don't know if $\partial(.,U)$ is continuous but I doubt it. Is this worth pointing out?}

\begin{lem}\label{l:caraccompact2}
Assume that $(a_n)$ is a sequence in $S(X)$ and $\delta > 0$ is such that $\partial(a_n, a_m) > \delta$ for all $n \neq m$. Then, for every $\varepsilon < \delta/2$, there exist a subsequence $(b_n)$ of $(a_n)$ and $\tau$-open sets $(U_n)$ such that $b_n \in U_n$ and $\partial(U_n,U_m) \ge \varepsilon$ for all $n \ne m$.
\end{lem}
\begin{proof}
Let $f \in \cli(X)$ be such that $|f(a_0)-f(a_1)| > \delta$. The triangle inequality implies that, for all $n>1$, we have  $|f(a_0)-f(a_n)|> \frac{\delta}{2}$ or $|f(a_1)-f(a_n)| > \frac{\delta}{2}$. One of those cases happens infinitely many times. Thus we see that for any such sequence $(a_n)$, there exists $i_0 \in \{0,1\}$, an infinite subset $\{i_n\}_{n \ge 1} \subseteq \bN \setminus \{0,1\}$ and $f_0 \in \cli(X)$ such that $f_0(a_{i_0})=0$ and $f_0(a_{i_n}) > \frac{\delta}{2}$ for all $n \ge 1$. Repeating this infinitely many times, we build a subsequence $(b_n)$ of $(a_n)$ and a sequence of maps $f_n \in \cli(X)$ such that $f_n(b_n)=0$ for all $n$ and $f_n(b_m) > \frac{\delta}{2}$ for all $n < m$.

Set $U_n= \{a \in S(X) \colon f_n(a) < \frac{\delta}{2} - \varepsilon \text{ and } f_k(a) > \frac{\delta}{2} \text{ for all } k <n \}$.
We have $b_n \in U_n$ for all $n$, and the function $f_n$ witnesses the fact that $\partial(U_n,U_m) \ge \varepsilon$ for all $n<m$.
\end{proof}

\begin{proof}[Proof of Theorem \ref{t:caraccompact}]
  Let $(a_n)$ be a $\tau$-convergent sequence in $S(X)$ with limit $a$ and suppose that it does not admit a $\partial$-Cauchy subsequence. Then there exists $\delta > 0$ such that $\partial(a_n, a_m) > \delta$ for all $n \neq m$ and we can apply Lemma~\ref{l:caraccompact2} to obtain a subsequence $(b_n)$ of $(a_n)$ and $\tau$-open subsets $U_n$ of $S(X)$ such that $b_n \in U_n$ and $\partial(U_n,U_m) \ge \delta/4$ for all $n \ne m$. Let
\[
  U= \bigcup_{n} U_{2n} \quad \text{ and } \quad V = \bigcup_{n} U_{2n+1}.
\] 
Then we have both that $\partial(U,V) \ge \delta/4$ and $a \in \cl[\tau]{U} \cap \cl[\tau]{V}$, which contradicts Lemma~\ref{l:caraccompact1}.

Thus every $\tau$-convergent sequence admits a $\partial$-Cauchy subsequence, which combined with the facts that $\partial$ is complete and that the $\partial$-topology refines $\tau$ implies the statement of the theorem.
\end{proof}

\begin{cor}
  \label{c:metrizable-subsp}
  Let $K \subseteq S(X)$ be a subset such that $K$ equipped with the relative $\tau$-topology is a metrizable topological space. Then the $\partial$-topology and $\tau$ coincide on $K$ and in particular, if $K$ is $\tau$-closed, $(K, \partial)$ is compact.
\end{cor}
\begin{proof}
  We already know that the $\partial$-topology is finer than $\tau$. To see the converse, note that if $K$ is metrizable, its topology is determined by convergence of sequences and then apply Theorem~\ref{t:caraccompact}.
\end{proof}

%%%%%%%%%%%%%%%%%%%%%%%%%%%%%%%%%%%%%%%%%%%%%%%%%%

\section{Proof of the main result}
Let $G$ be a Polish group, that is, a completely metrizable, separable topological group. 
Recall that the \emph{right uniformity} on $G$ is given by the basis of entourages $\{(g,h) \colon gh^{-1} \in U\}$, where $U$ ranges over all neighbourhoods of $1_G$; by the Birkhoff--Kakutani theorem, there exists a compatible right-invariant distance on $G$ and any such distance must induce the right uniformity. We fix a bounded, right-invariant distance $\dR$ on $G$. Then $(G, \dR)$ is a metric space and we construct its topometric Stone--\v{C}ech compactification $(S(G), \tau, \partial)$ as in the previous section.

In this setting, the universal property of $S(G)$ translates to the following: if $G \actson X$ is a $G$-flow and $x_0 \in X$, then there exists a unique $G$-map $\pi \colon S(G) \to X$ such that $\pi(1_G) = x_0$. It is then clear that whenever $M \subseteq S(G)$ is a minimal subflow, $M$ is \emph{universal}, that is, $M$ equivariantly maps onto every other minimal $G$-flow. One important property of (any such) $M$ that we will need is that it is \emph{coalescent}, i.e., any endomorphism of it is an automorphism. This fact is due to Ellis (see Proposition 3.3 in \cite{Uspenskij2000} for a simple proof). This implies that any two universal minimal flows must be isomorphic, which allows us to speak of \emph{the} universal minimal flow $M(G)$ of $G$.

We will need the following folklore lemma, which follows from standard facts on the enveloping semigroup. Recall that if $Z$ is a topological space, a continuous map $r \colon Z \to Z$ is called a \emph{retraction} if $r(r(z)) = r(z)$ for all $z \in Z$. 
\begin{lem}\label{l:retraction}
Let $G$ be a topological group, let $S(G)$ be its Samuel compactification, and let $M$ be a minimal subflow of $S(G)$. Then there exists a $G$-equivariant retraction 
$r \colon S(G) \to M$.
\end{lem}
\begin{proof}
The universal property of $S(G)$ gives us a $G$-equivariant map $f \colon S(G) \to M$. The restriction of $f$ to $M$ is an endomorphism of $M$ and is thus, by coalescence, an automorphism of $M$; denote this automorphism by $\varphi$. Then $\varphi^{-1} \circ f$ is a retraction from $S(G)$ onto $M$.
\end{proof} 

% Our aim in this article is to gain a better understanding of the situation when $M(G)$ is metrizable; as explained in the introduction, we want to prove that $M(G)$ must then have a comeagre orbit. To that end, we will apply a criterion for the existence of such an orbit for a topologically transitive action \textcolor{red}{REF?}; but we must first gain a better understanding of metrizable compact subsets of $M(G)$, which we do in the next section, via an approach based on the theory of \emph{topometric spaces} (see \cite{BenYaacov2008c} for more details on those spaces; everything that is needed here is recalled in the next section).

To prove the existence of a comeagre orbit in Theorem~\ref{t:main2}, we use the following criterion, due to C.~Rosendal.

\begin{prp}
  \label{p:comeagre-orbit}
  Let $X$ be a Polish space and $G$ be a Polish group acting on $X$ continuously and topologically transitively. Then the following are equivalent:
  \begin{enumerate}
  \item \label{i:co:1} There exists a comeagre $G$-orbit;

  \item \label{i:co:2} For any open $V \ni 1_G$ and any non-empty open subset $U$ of $X$, there exists a non-empty open $U'\subseteq U$ such that for any non-empty open $W_1,W_2 \subseteq U'$, we have $V \cdot W_1 \cap W_2 \ne \emptyset$.
  \end{enumerate}
\end{prp}
\begin{proof}
  \ref{i:co:1} $\Rightarrow$ \ref{i:co:2}. Suppose that $G \cdot x_0 \sub X$ is a comeagre orbit and let $V \ni 1_G$ and $U \sub X$ be given. Let $1_G \ni V' \sub G$ be open such that $V' V'^{-1} \sub V$. Let $x \in U \cap G \cdot x_0$. By Effros's theorem, there exists an open $U_0 \sub X$ such that $V' \cdot x = U_0 \cap G \cdot x$, and in particular, $U_0 \sub \cl{V' \cdot x}$. Finally, let $U' = U_0 \cap U$ and let $W_1, W_2 \sub U'$. Then there exist $v_1, v_2 \in V'$ such that $v_1 \cdot x \in W_1$ and $v_2 \cdot x = v_2 v_1^{-1} v_1 \cdot x \in V \cdot W_1 \cap W_2$.

  \ref{i:co:2} $\Rightarrow$ \ref{i:co:1}. Suppose that there is no comeagre orbit. Since we are assuming that the action is topologically transitive, the topological zero--one law implies that all orbits are meagre. Let $x \in X$ be arbitrary. As $G \cdot x$ is meagre, there exist closed, nowhere dense sets $\{F_n : n \in \bN\}$ such that $G \cdot x \sub \bigcup_n F_n$. Let $B_n = \{g \in G : g \cdot x \in F_n\}$. Then each $B_n$ is closed and by the Baire category theorem, there exists $n$ such that $B_n$ has non-empty interior. Let $V \ni 1_G$ be open and $g \in G$ be such that $gV \sub B_n$. Then $\cl{V \cdot x} \sub g^{-1} \cdot F_n$, so $V \cdot x$ is nowhere dense. We conclude that for every point $x \in X$, there exists $V \ni 1_G$ such that $V \cdot x$ is nowhere dense. As there are only countably many possible $V$ (we can always restrict to a basis at $1_G$), there exists $V \ni 1_G$ and a non-empty, open set $U \sub X$ such that the set
  \[
    \{x \in U : V \cdot x \text{ is somewhere dense}\}
  \]
  is meagre. (Here we are using the fact that $\{x \in X : V \cdot x \text{ is nowhere dense}\}$ is a Borel set and hence has the property of Baire, and that a non-meagre set with the Baire property must be comeagre in an open set.) Let now $U' \sub U$ be as given by \ref{i:co:2}. The property of $U'$ implies that the set 
  \begin{equation*}
    \{ x \in U' : V \cdot x \text{ is dense in } U' \} = \{x \in U' : \forall W_1 \sub U \ V \cdot x \cap W_1 \neq \emptyset \}
  \end{equation*}
  is dense $G_\delta$ in $U'$, which is a contradiction (the quantifier $\forall W_1 \sub U$ can be taken to range over a countable basis of $U$).
\end{proof}

We are now ready to prove our main result.
\begin{proof}[Proof of Theorem \ref{t:main2}]
We view $M(G)$ as a subflow of $S(G)$ and fix a $G$-equivariant retraction $r \colon S(G) \to M(G)$, as given by Lemma~\ref{l:retraction}. To show that there is a comeagre orbit, we will apply Proposition~\ref{p:comeagre-orbit}. Let $V \ni 1_G$ and $U \sub M(G)$ be given. We may assume that $V = \{g : \dR(g,1_G) < \varepsilon\}$ for some $\varepsilon > 0$. Since $M(G)$ is metrizable, we have that $\tau$ and the $\partial$-topology coincide on $M(G)$ by Corollary~\ref{c:metrizable-subsp} and we can find a non-empty $\tau$-open $U' \subseteq U$ of $\partial$-diameter $< \varepsilon$. Let $W_1, W_2 \subseteq U'$ be non-empty, open. By the choice of $U'$, we have $\partial(W_1,W_2) < \varepsilon$; since $W_1 \subseteq r^{-1}(W_1)$ and $W_2 \subseteq r^{-1}(W_2)$, we also have that $\partial(r^{-1}(W_1), r^{-1}(W_2))< \varepsilon$. Then Lemma~\ref{l:caraccompact1} tells us that $\dR(r^{-1}(W_1) \cap G, r^{-1}(W_2) \cap G) < \varepsilon$. So we can find $f_1 \in r^{-1}(W_1) \cap G$ and $f_2 \in r^{-1}(W_2) \cap G$ such that $\dR(f_1, f_2) < \varepsilon$, that is, $f_2 f_1^{-1} \in V$. Since $r(f_2) = f_2f_1^{-1}r(f_1) \in W_2 \cap f_2f_1^{-1}W_1$, the criterion is verified and we are done.
\end{proof}
Note that even though our proof produces a ``natural'' metric $\partial$ on $M(G)$ compatible with the topology, this metric is not canonical: it depends on the right-invariant metric on $G$ that we start with, and more importantly, on the copy of $M(G)$ in $S(G)$ that we consider.

Theorem~\ref{t:main1} follows directly from Theorem~\ref{t:main2} and \cite[Theorem~1.2]{Melleray15}.

We close the paper by pointing out several consequences of Theorem~\ref{t:main1}.
\begin{cor}
  \label{c:minimal}
  Suppose that $G$ is a Polish group which admits a closed, extremely amenable, co-precompact subgroup $H$. Then there exists a subgroup $H' \leq G$ which is still extremely amenable and co-precompact, and, moreover $\widehat{G/H'}$ is minimal.
\end{cor}
\begin{proof}
  The flow $G \actson \widehat{G/H}$ is metrizable and it maps to any $G$-flow (Pestov~\cite[Lemma~2.3]{Pestov2002a}; see also \cite{Melleray15}); any of its minimal subflows is isomorphic to $M(G)$, which is therefore metrizable, and thus Theorem~\ref{t:main1} applies.
\end{proof}

Recall that a Polish group $G$ is called a \emph{CLI group} if its right uniformity is complete (equivalently, admits a right-invariant, complete, compatible metric) and a \emph{SIN group} if its left and right uniformities coincide (equivalently, admits a bi-invariant compatible metric). Every Polish SIN group is CLI.
\begin{cor}
  \label{c:SIN-groups}
  Assume that $G$ is a Polish SIN group and $M(G)$ is metrizable. Then $M(G)$ is a compact group, that is, there exists a \emph{normal}, extremely amenable, closed subgroup $G^*$ of $G$ such that $M(G)=G/G^*$. In particular, $G$ is amenable.
\end{cor}
\begin{proof}
  Theorem~\ref{t:main1} gives us an extremely amenable subgroup $G^*$ such that $M(G)= \widehat{G/G^*}$. Pick a compatible bi-invariant metric $d$ on $G$. Such a metric must be complete, so in that case, $G/G^*$ is also complete and we have $M(G)=G/G^*$. The action of $G$ on $G/G^*$ is isometric for the natural quotient metric on $G/G^*$, giving us a homomorphism $\varphi$ from $G$ to the compact group $\Iso(G/G^*)$. Then $\overline{\varphi(G)}$ is a compact subgroup of $\Iso(G/G^*)$, on which $G$ acts minimally (via the left-translation action of $\varphi(G)$). The evaluation map $\pi \colon \cl{\varphi(G)} \to G/G^*$, $\pi(f) = f(G^*)$ is continuous and $G$-equivariant from $\overline{\varphi(G)}$ to $G/G^*$. Since $G/G^*=M(G)$, it must be the case that $\pi$ is an isomorphism, which can only happen if $\varphi(G)$ is closed. Hence $G/G^* \cong \varphi(G)$ is a group. In particular, $M(G)$ carries a $G$-invariant measure, the Haar measure on $G/G^*$, and $G$ is amenable.
\end{proof}
\begin{rmk}
  In \cite{Angel2014}, Angel, Kechris, and Lyons ask whether for an amenable, Polish group $G$, the metrizability of $M(G)$ implies that it is uniquely ergodic (if it is, then so is every minimal $G$-flow). While we cannot answer this question in general, we note that Corollary~\ref{c:SIN-groups} and the uniqueness of the Haar measure imply that this is true for SIN groups.
\end{rmk}

\begin{rmk}
  In \cite{Glasner1998a}, Glasner asks whether every monothetic Polish group that admits no non-trivial homomorphisms to the circle has to be extremely amenable. We note that Corollary~\ref{c:SIN-groups} implies a positive answer in the case where $M(G)$ is metrizable. This was also observed by L.~Nguyen~Van~Thé (see the forthcoming paper \cite{Nguyen2017}).
\end{rmk}

\begin{rmk}
  Another interesting family of groups to which Corollary~\ref{c:SIN-groups} applies are the full groups of countable, Borel, measure-preserving, ergodic equivalence relations. (See, for example, \cite[Chapter~I, 3]{Kechris2010} for the definition and their basic properties.) If $E$ is hyperfinite, then the full group $[E]$ is extremely amenable and if it is not, then $[E]$ is not amenable (Giordano--Pestov~\cite{Giordano2007}), and therefore, by Corollary~\ref{c:SIN-groups}, its universal minimal flow is not metrizable. This suggests that for non-hyperfinite $E$, the category of minimal flows of $[E]$ is quite rich: it would be interesting to investigate how its structure reflects the structure of the equivalence relation $E$.
\end{rmk}

The argument used at the beginning of the proof of Corollary~\ref{c:SIN-groups} to show that $M(G)=G/G^*$ only uses the fact that $\dR$ is complete, so we also have the following.
\begin{cor}
  \label{c:CLI}
Assume that $G$ is a Polish CLI group and that $M(G)$ is metrizable. Then $G \actson M(G)$ is transitive and so is any other minimal $G$-flow.
\end{cor}

We do not know an example of a CLI group $G$ which is not SIN and such that $M(G)$ is metrizable.

\begin{cor}
  Let $G$ be a Polish group such that $M(G)$ is metrizable and the action $G \actson M(G)$ is free. Then $G$ is compact.
\end{cor}
\begin{proof}
  Theorem~\ref{t:main1} gives us a co-precompact $G^* \leq G$ such that $M(G) = \widehat{G/G^*}$; as the action $G \actson M(G)$ is free, this implies that $G^* = \{1_G\}$, which, in turn, means that $G$ is compact.
\end{proof}
Veech's well-known theorem that if $G$ is locally compact, then the action $G \actson S(G)$ is free now allows us to recover the following result of \cite{Kechris2005}: if $G$ is Polish, locally compact and $M(G)$ is metrizable, then $G$ is compact.

\end{document}